\documentclass[a4paper,12pt, reqno]{amsart}

\usepackage{amsmath, amsthm, amscd, amsfonts, amssymb, graphicx, color}
\usepackage[bookmarksnumbered, colorlinks, plainpages]{hyperref}

\usepackage{enumitem}
\usepackage{url}
\usepackage{multirow}
\usepackage{graphicx}
\usepackage{subfigure}
\usepackage[pass]{geometry}
\usepackage{cite}
\usepackage{cancel}
\usepackage{color}
\usepackage{pgf,tikz}
\usetikzlibrary{arrows}
\usepackage{lscape}
\tikzstyle{v} = [circle, draw, inner sep=2pt, minimum size=3pt, fill=black]
\usetikzlibrary{calc,shapes,decorations.pathreplacing,matrix}
\usetikzlibrary{patterns}

\tikzset{square matrix/.style={
    matrix of nodes,
    column sep=-\pgflinewidth, row sep=-\pgflinewidth,
    nodes={draw,
      minimum height=4.5pt,
      anchor=center,
      text width=4.5pt,
      align=center,
      inner sep=0pt
    },
  },
  square matrix/.default=1.2cm
}

\newtheorem{Theorem}{Theorem}[section]
\newtheorem{Definition}[Theorem]{Definition}

\newtheorem{Lemma}[Theorem]{Lemma}
\newtheorem{Proposition}[Theorem]{Proposition}
\newtheorem{Corollary}[Theorem]{Corollary}
\newtheorem{Remark}[Theorem]{Remark}
\newtheorem{Example}[Theorem]{Example}

\DeclareMathOperator{\leaf}{leaf}
\DeclareMathOperator{\diam}{diam}
\DeclareMathOperator{\cor}{cor}

\begin{document}

\title{Domination number of middle graphs}

\author[F. Kazemnejad]{Farshad Kazemnejad}
\address{Farshad Kazemnejad, Department of Mathematics, School of Sciences, Ilam University, 
P.O.Box 69315-516, Ilam, Iran.}
\email{kazemnejad.farshad@gmail.com}
\author[B. Pahlavsay]{Behnaz Pahlavsay}
\address{Behnaz Pahlavsay, Department of Mathematics, Hokkaido University, Kita 10, Nishi 8, Kita-Ku, Sapporo 060-0810, Japan.}
\email{pahlavsay@math.sci.hokudai.ac.jp}
\author[E. Palezzato]{Elisa Palezzato}
\address{Elisa Palezzato, Department of Mathematics, Hokkaido University, Kita 10, Nishi 8, Kita-Ku, Sapporo 060-0810, Japan.}
\email{palezzato@math.sci.hokudai.ac.jp}
\author[M. Torielli]{Michele Torielli}
\address{Michele Torielli, Department of Mathematics, GI-CoRE GSB, Hokkaido University, Kita 10, Nishi 8, Kita-Ku, Sapporo 060-0810, Japan.}
\email{torielli@math.sci.hokudai.ac.jp}

\date{\today}

\begin{abstract}
 In this paper, we study the domination number of middle graphs. Indeed, we obtain tight bounds for this number in terms of the order of the graph $G$. We also compute the domination number of some families of graphs such as star graphs, double start graphs, path graphs, cycle graphs, wheel graphs, complete graphs, complete bipartite graphs and friendship graphs, explicitly. Moreover, some Nordhaus-Gaddum-like relations are presented for the domination number of middle graphs.
\\[0.2em]

\\[0.2em]

\noindent 
\end{abstract}
\maketitle
\section{\bf Introduction}

The notion of domination and its many generalizations have been intensively studied in graph theory and the literature on this subject is vast, see for example \cite{HHS5}, \cite{HHS6}, \cite{HeYe13}, \cite{3totdominrook} and \cite{dominLatin}. Throughout this paper, we use standard notation for graphs and we assume that each graph is non-empty, finite, undirected and simple. We refer to \cite{bondy2008graph} as a general reference on graph theory.


Let $G$ be a graph with vertex set $V(G)$ of \emph{order}
$n$ and edge set $E(G)$ of \emph{size} $m$.
The \emph{open neighborhood} of a
vertex $v\in V(G)$ is $N_{G}(v)=\{u\in V(G)~|~ uv\in E(G)\}$ and, similarly, the \emph{closed neighborhood} of a
vertex $v\in V(G)$ is $N_{G}[v]=N_{G}(v)\cup \{v\}$. The \emph{degree} of a
vertex $v\in V(G)$ is defined as $d_G(v)=\vert N_{G}(v) \vert $. 
The \emph{distance} $d_G(v,w)$ in $G$ of two vertices $v,w\in V(G)$ is the length of the shortest path connecting $v$ and $w$.
The \emph{diameter} of $G$, denoted $\diam(G)$, is the shortest distance between any two vertices in $G$.

\begin{Definition} A \emph{dominating set}, briefly DS,  of a graph $G$ is a set $S\subseteq V(G)$
such that  $N_G[v]\cap
S\neq \emptyset$, for any vertex $v\in V(G)$. The \emph{domination number} of $G$ is the minimum cardinality of a DS of $G$ and it is denoted by $\gamma(G)$. 
\end{Definition}

For any non-empty $S\subseteq V(G)$, we denote by $G[S]$ the subgraph of $G$ induced on the vertex set $S$. For any $v\in V(G)$, we denote by $G\setminus v$ the subgraph of $G$ induced on the vertex set $V(G)\setminus \{v\}$. Given two graphs $G$ and $H$ with distinct vertices, we can construct a new graph $G\cup H$ by imposing $V(G\cup H)=V(G)\cup V(H)$ and $E(G\cup H)=E(G)\cup E(H)$. 

Given a graph $G$, its \emph{complement}, denoted by $\overline{G}$, is a graph with vertex set $V(G)$ such that for every two vertices $v$ and $w$, $vw\in E(\overline{G})$\ if and only if $vw\not\in E(G)$. 

The \emph{line graph} $L(G)$ of a graph $G$ is the graph with vertex set $E(G)$, where vertices $x$ and $y$ are adjacent in $L(G)$ if and only if the corresponding edges $x$ and $y$ share a common vertex in $G$.

The concept of middle graph $M(G)$ of a graph $G$ was introduced by Hamada and Yoshimura in \cite{HamYos} as an intersection graph on the vertex set of $G$.

\begin{Definition}
 The middle graph $M(G)$ of a graph $G$ is the graph whose vertex set is $V(G)\cup E(G)$ and two vertices $x, y$ in the vertex set of $M(G)$ are adjacent in $M(G)$ in case one the following holds
 \begin{enumerate}
 \item $x, y\in E(G)$ and $x, y$ are adjacent in $G$;
 \item $x\in V (G)$, $y\in E(G)$, and $x, y$ are incident in $G$. 
 \end{enumerate}
 \end{Definition}
 
In other words, the middle graph $M(G)$ of a graph $G$ of order $n$ and size $m$ is a graph of order $n+m$ and size $2m+|E(L(G))| $ which is obtained by subdividing each edge of $G$ exactly once and joining all the adjacent edges of $G$ in $ M(G)$. 
It is obvious that $M(G)$ always contains the line graph $L(G)$ as an induced subgraphs.

In order to avoid confusion throughout the paper, we fix a ``standard'' notation for the vertex set and the edge set of the middle graph $M(G)$. Assume $V(G)=\{v_1,v_2,\dots, v_n\}$, then we set $V(M(G))=V(G)\cup \mathcal{M}$, where $\mathcal{M}=\{m_{ij}~|~ v_iv_j\in E(G)\}$ and $E(M(G))=\{v_im_{ij},v_jm_{ij}~|~ v_iv_j\in E(G)\}\cup E(L(G)) $.


The paper proceeds as follows. In Section 2, first we present some upper and lower bounds for $\gamma(M(G))$ in terms of the order of the graph $G$, we relate the domination number of $M(G)$ to the edge cover number of $G$, and then we compute the domination number of the middle graphs of the corona, $2$-corona and the join of graphs. In Section 3, we compute explicitly $\gamma(M(G))$ for several known families of graphs: star graphs, double star graphs, path graphs, cycle graphs, wheel graphs, complete graphs, complete bipartite graphs and friendship graphs. Finally, in the last Section we present some Nordhaus-Gaddum like relations for the domination number of middle graphs.

\section{\bf Domination number of middle graphs}
We start our study of domination numbers of middle graph with two key Lemmas.

\begin{Lemma}\label{lemma:dominationisalledges}
Let $G$ be a graph of order $n\ge2$ without isolated vertices and $S$ a dominating set of $M(G)$. Then there exists $S'\subseteq E(G)$ a dominating set of $M(G)$ with $|S'|\le|S|$.
\end{Lemma}
\begin{proof} If $S\subseteq E(G)$, then take $S'=S$. On the other hand, assume that there exists $v\in S\cap V(G)$. If all incident edges to $v$ are already in $S$, then take $S_1=S\setminus\{v\}$. Otherwise, let $e\in E(G)\setminus S$ an edge incident to $v$. Then consider $S_1=(S\cup\{e\})\setminus\{v\}$. By construction, $S_1$ is again a dominating set of $M(G)$. Since $S$ is finite, then this process terminates after a finite number of steps, and hence we obtain the described $S'$.

\end{proof}

\begin{Lemma}\label{lemma:dominationdeletionvertex}
Let $G$ be a graph of order $n\ge2$ and $v\in V(G)$. Then
$$\gamma(M(G\setminus v))\le \gamma(M(G)) \le \gamma(M(G\setminus v))+1.$$ 
\end{Lemma}
\begin{proof} For any dominating set $S$ of $M(G\setminus v)$, we have that $S\cup\{v\}$ is a dominating set of $M(G)$, and hence $\gamma(M(G)) \le \gamma(M(G\setminus v))+1$.

On the other hand, let $S$ be a minimal dominating set of $M(G)$. If $v$ is an isolated vertex, then $v\in S$ and $S\setminus\{v\}$ is a minimal dominating set of $M(G\setminus v)$. This implies that in this case $\gamma(M(G))= \gamma(M(G\setminus v))+1$. Assume that $G$ has no isolated vertices. By Lemma~\ref{lemma:dominationisalledges}, we can assume that $S\subseteq E(G)$. Consider $S_v=N_{M(G)}(v)\cap S$. Since $S$ is a dominating set, $|S_v|\ge 1$. Assume $S_v=\{e_1,\dots, e_k\}$. For any $1\le i\le k$, $e_i$ is an edge of $G$ of the form $w_iv$. Define $S'=(S\setminus S_v)\cup\{w_1,\dots, w_k\}$. By construction $S'$ is a dominating set of $M(G\setminus v)$ with $|S'|= |S|$, and hence $\gamma(M(G\setminus v))\le \gamma(M(G))$.

\end{proof}

We start our study of the domination number by describing a lower and an upper bound for the domination number of the middle graph of a tree.

\begin{Theorem}\label{theo:mindomintree}
Let $T$ be a tree with $n\ge2$ vertices. Then $$\lceil\frac{n}{2}\rceil\le\gamma(M(T))\le n-1.$$
\end{Theorem}
\begin{proof} Let  $S$ be a dominating set of $M(T)$. By Lemma~\ref{lemma:dominationisalledges}, we can assume that $S\subseteq E(T)$. This implies that $|S|\le|E(T)|=n-1$, proving the second inequality.

Assume that there exists $S$ a dominating set of $M(T)$ with $|S|<\lceil\frac{n}{2}\rceil$. By Lemma~\ref{lemma:dominationisalledges}, we can assume that $S\subseteq E(T)$, this implies that $V_S$ the set of vertices of $T$ that are incident to at least one edge in $S$ satisfies $|V_S|<n$, and hence there exists at least one $v\in V(T)\subseteq V(M(T))$ not dominated by $S$, proving the first inequality.

\end{proof}

If we denote by $\leaf(T)=\{v\in V(T)~|~d_T(v)=1\}$ the set of leaves of a tree $T$, then we have the following result.
\begin{Proposition}\label{prop:mindomintreeleaf}
Let $T$ be a tree with $n\ge2$ vertices. Then $$\gamma(M(T))\ge |\leaf(T)|.$$
\end{Proposition}
\begin{proof} To fix the notation, assume $\leaf(T)=\{v_1,\dots,v_k\}$, for some $k\le n$, and let $S$ be a dominating set of $M(T)$. Then, for each $i=1,\dots, k$, $S\cap N_{M(T)}[v_i]\ne\emptyset$. Since, if $i\ne j$, then $N_{M(T)}[v_j]\cap N_{M(T)}[v_i]=\emptyset$, we have that $|S|\ge k$. This implies that $\gamma(M(T))\ge k= |\leaf(T)|.$

\end{proof}

If we specialize the class of trees that we are considering, we obtain an exact value for the domination number. 
\begin{Theorem}\label{theomindomintreediam3}
Let $T$ be a tree of order $n\geq 4$ with $\diam(T)=3$. Then $$\gamma(M(T))=n-2.$$
\end{Theorem}
\begin{proof} 
Since by assumption $\diam(T)=3$, then $T$ is a tree which is obtained by joining with an extra edge the central vertex $v_{n-1}$ of a tree $K_{1,p}$ and the central vertex $v_n$ of a tree $K_{1,q}$ where $p+q=n-2$. 

Let $\leaf(K_{1,p})=\{v_i~|~1\le i\le p\}$, $\leaf(K_{1,q})=\{v_i~|~p+1\le i\le p+q\}$ and $\leaf(G)=\leaf(K_{1,p})\cup\leaf(K_{1,q})$ be the sets of leaves of $K_{1,p}$, $K_{1,q}$ and $G$, respectively. Obviously, $|\leaf(G)|=n-2$ and $V(T)=\leaf(G)\cup \{v_{n-1},v_n\}$. 

Let $S$ be a dominating set of $M(T)$. Since $N_{M(G)}[v_i]\cap N_{M(G)}[v_j]= \emptyset$ for every distinct $v_i,v_j\in \leaf(G)$, we have $|S|\geq n-2$, and hence $\gamma(M(T))\ge n-2$.

On the other hand, if we consider $S=\{m_{i(n-1)}~|~1\leq i \leq p \}\cup\{m_{in}~|~p+1\leq i \leq n-2\}$, then $S$ is a dominating set of $M(T)$ with $|S|=n-2$, and hence $\gamma(M(T))\le n-2$. 

\end{proof}

Since path graphs are special type of trees, in general for a tree $T$, $\gamma(M(T))=n-2$ does not imply $\diam(T)=3$, as the next example shows.
\begin{Example} Consider the path graph $P_5$. Then $\diam(P_5)=4$ and $\gamma(M(P_5))=3=n-2$.
\end{Example}

Next we describe a lower and an upper bound for the domination number of the middle graph of an arbitrary graph.

\begin{Theorem}\label{theo:mindominconnectgraph}
Let $G$ be a graph with $n\ge2$ vertices. Assume $G$ has no isolated vertices,  then $$\lceil\frac{n}{2}\rceil\le\gamma(M(G))\le n-1.$$
\end{Theorem}
\begin{proof} Assume first that $G$ is connected. Let $T$ be a spanning tree of $G$, and $S$ a minimal dominating set of $M(T)$. By Theorem~\ref{theo:mindomintree}, $|S|\le n-1$, and by Lemma~\ref{lemma:dominationisalledges}, we can assume that $S\subseteq E(T)\subseteq E(G)$. By construction, $S$ is also a dominating set of $M(G)$, and hence $\gamma(M(G))\le n-1$.
Assume that $G$ has $k+1\ge2$ connected components. Then we can apply the previous argument to each connected component and obtain $\gamma(M(G))\le n-1-k$.

To prove the first inequality, assume that there exists $S$ a dominating set of $M(G)$ with $|S|<\lceil\frac{n}{2}\rceil$. By Lemma~\ref{lemma:dominationisalledges}, we can assume that $S\subseteq E(G)$, this implies that $V_S$ the set of vertices of $G$ that are incident to at least one edge in $S$ satisfies $|V_S|<n$, and hence there exists at least one $v\in V(G)\subseteq V(M(G))$ not dominated by $S$, proving the first inequality.
\end{proof}

\begin{Remark} The proof of Theorem~\ref{theo:mindominconnectgraph} shows that if $G$ is a graph with $k+1$ connected components, then $\gamma(M(G))\le n-1-k$.
\end{Remark}

\begin{Definition}
An \emph{edge cover} of a graph $G$ is a set of edges $S\subseteq E(G)$ such that every vertex of $G$ is incident to at least one edge in $S$.
The \emph{edge cover number} of $G$, denoted by $\rho(G)$, is the minimum cardinality of an edge cover of $G$.
\end{Definition}

Using Lemma~\ref{lemma:dominationisalledges}, we can relate the domination number of a middle graph and the edge cover number of the original graph.

\begin{Theorem}\label{theo:mindomandcovernumber}
Let $G$ be a graph of order $n\geq 2$ with no isolated vertex. Then $$\gamma(M(G))= \rho(G).$$
\end{Theorem}
\begin{proof} 
To fix the notation, assume $V(G)=\{v_1,\dots, v_n\}$. Let $S$ be a minimum edge cover of $G$. Then $N_{M(G)}(v_i)\cap S\neq \emptyset$ for all $1\leq i \leq n$. This implies that for all $1\leq i \leq n$, there exist an index $k\ne i$ such that $m_{ik}\in N_{M(G)}(v_i)\cap S$, and hence $ N_{M(G)}[m_{rs}]\cap S\neq \emptyset$ for every $m_{rs}\in V(M(G))$. This implies that $S$ is a dominating set of $M(G)$, and hence that  $\gamma(M(G))\leq \rho(G)$.

On the other hand, let $S$ be a minimal dominating set of $M(G)$. By Lemma~\ref{lemma:dominationisalledges}, we can assume that $S\subseteq E(G)$. Since $S$ dominates every vertex of $G$, we have $N_{M(G)}(v_i)\cap S\neq \emptyset$, for every $1\leq i \leq n$. This implies that $S$ is an edge cover of $G$, and hence $\gamma(M(G))\geq \rho(G)$.

\end{proof}

\begin{Definition}
The \emph{corona} $G\circ K_1$ (also denoted by $\cor(G)$) of a graph $G$ is the graph of order $2|V(G)|$ obtained from $G$ by adding a pendant edge to each vertex of $G$. The \emph{$2$-corona} $G\circ P_2$ of $G$ is the graph of order $3|V(G)|$ obtained from $G$ by attaching a path of length $2$ to each vertex of $G$ so that the resulting paths are vertex-disjoint.
\end{Definition}

\begin{Theorem}\label{theo:mindomincorona}
For any connected graph $G$ of order $n\geq 2$, $$\gamma(M(G\circ K_1))=n.$$
\end{Theorem}
\begin{proof} 
To fix the notation, assume $V(G)=\{v_1,\dots, v_{n}\}$. Then $V(G\circ K_1)= \{v_{1},\dots, v_{2n}\}$ and $E(G\circ K_1)=\{v_1v_{n+1},\dots, v_nv_{2n} \}\cup E(G) $. Then $V(M(G\circ K_1))=V(G\circ K_1)\cup \mathcal{M}$, where $\mathcal{M}=\{ m_{i(n+i)}~|~1\leq i \leq n \}\cup \{ m_{ij}~|~v_iv_j\in  E(G)\}$.

If we consider $S=\{m_{i(n+i)}~|~1\leq i \leq n\}$, then $S$ is a dominating set of $M(G\circ K_1)$ with $|S|=n$, and hence $\gamma(M(G\circ K_1))\le n$. 
On the other hand, by Theorem~\ref{theo:mindominconnectgraph}, $\gamma(M(G\circ K_1))\ge \lceil\frac{2n}{2}\rceil=n$.

\end{proof}

\begin{Theorem}\label{theomindomin2corona}
For any connected graph $G$ of order $n\geq 2$, $$\gamma(M(G\circ P_2))=n+\gamma(M(G)).$$
\end{Theorem}
\begin{proof} 
To fix the notation, assume $V(G)=\{v_1,\dots, v_n\}$. Then $V(G\circ P_2)= \{v_{1},\dots, v_{3n}\}$ and $E(G\circ P_2)=\{v_iv_{n+i}, v_{n+i}v_{2n+i}~|~1\leq i \leq n \}\cup E(G) $. Then $V(M(G\circ P_2))=V(G\circ P_2)\cup \mathcal{M}$, where $\mathcal{M}=\{ m_{i(n+i)},m_{(n+i)(2n+i)}~|~1\leq i \leq n \}\cup \{ m_{ij}~|~v_iv_j\in  E(G)\}$.

Let $S$ be a dominating set of $M(G\circ P_2)$. By Lemma~\ref{lemma:dominationisalledges}, we can assume that $S\subseteq\mathcal{M}$. 
Since $d_{G\circ P_2}(v_{2n+i})=1$, for every $1\leq i\leq n$, we have $m_{(n+i)(2n+i)}\in S$ for every $1\leq i\leq n$. In addition, since $N_{M(G)}[v_i]\cap\{m_{(n+j)(2n+j)}~|~1\le j\le n\}=\emptyset$, for all $1\leq i\leq n$, we have $|S|\geq n+\gamma(M(G))$.

On the other hand, if $S'$ is a minimal dominating set of $M(G)$, then $S=S'\cup\{m_{(n+i)(2n+i)}~|~1\le i\le n\}$ is a dominating set of $M(G\circ P_2)$ with $|S|=n+\gamma(M(G))$, and hence $\gamma(M(G\circ P_2))\le n+\gamma(M(G))$.

\end{proof}

\begin{Definition}
The \emph{join} $G+ H$ of two graphs $G$ and $H$ is the graph with vertex set $V(G+H)=V(G)\cup V(H)$ and edge set $E(G+H)=E(G)\cup E(H)\cup \{vw~|~v\in V(G), w\in V(H)\}$.
\end{Definition}

In the next series of theorems, we study the domination number of the middle graph of the join of a graph with $\overline{K_p}$. 
\begin{Theorem}\label{theomindominjoinpbig}
For any connected graph $G$ of order $n\geq 2$ and any integer $p\geq n$, 
$$\gamma(M(G+\overline{K_p}))= p.$$
\end{Theorem}
\begin{proof}
To fix the notation, assume $V(G)=\{v_1,\dots,v_n\}$ and $V(\overline{K_p})=\{v_{n+1},\dots,v_{n+p}\}$. Then $V(M(G+\overline{K_p}))=V(G+\overline{K_p})\cup  \mathcal{M}_1\cup  \mathcal{M}_2$ where $\mathcal{M}_1= \{m_{ij}~|~v_iv_j\in  E(G)\}$ and $\mathcal{M}_2= \{m_{i(n+j)}~|~1\leq i \leq n, 1\leq j \leq p\}$.  

Let $S$ be a dominating set of $M(G+\overline{K_p})$. By Lemma~\ref{lemma:dominationisalledges}, we can assume $S\subseteq \mathcal{M}_1\cup \mathcal{M}_2$. Since, if $j\ne k$, $N_{M(G+ \overline{K_p})}(v_{n+j})\cap N_{M(G+ \overline{K_p})}(v_{n+k})=\emptyset$, then for every $1\le j\le p$ there exists $1\le i\le n$ such that  $m_{i(n+j)}\in S$, and hence $|S|\ge p$. This implies that $\gamma(M(G+\overline{K_p}))\ge p$. 

On the other hand, if we consider $S=\{m_{i(n+i)}~|~1\leq i \leq n\}\cup \{m_{1(n+j)}~|~n+1\leq j \leq p\}$, then $S$ is a dominating set of $M(G+\overline{K_p})$ with $|S|=p$, and hence $\gamma(M(G+\overline{K_p}))\le p$.

\end{proof}

\begin{Theorem}\label{theo:mindominjoinpsmallineq}
For any connected graph $G$ of order $n\geq 2$ and any integer $p<n$, 
$$\lceil\frac{n+p}{2}\rceil\le\gamma(M(G+\overline{K_p}))\le n.$$
\end{Theorem}
\begin{proof} The first inequality follows directly from Theorem~\ref{theo:mindominconnectgraph}.
On the other hand, using the same notation as in the proof of Theorem~\ref{theomindominjoinpbig}, if we consider $S=\{m_{i(n+i)}~|~1\le i\le p\}\cup\{v_j~|~p+1\le j\le n\}$, then $S$ is a dominating set of $M(G+\overline{K_p})$ with $|S|=n$, and hence we obtain the second inequality.
\end{proof}

\begin{Remark} Both inequalities in Theorem~\ref{theo:mindominjoinpsmallineq} are sharp. In fact, if $G=C_4$ and $p=2$, then a direct computation shows that $\gamma(M(C_4+\overline{K_2}))=3=\lceil\frac{n+p}{2}\rceil$. Moreover, if $G=C_4$ and $p=3$, then $\gamma(M(C_4+\overline{K_3}))=4=n$. See also Corollary~\ref{coroll:mindominjoinfamilspecial}.
\end{Remark}

\begin{Theorem}\label{theo:mindominjoinpsmallequalmin}
For any connected graph $G$ of order $n\geq 2$ and any integer $p<n$, 
$$\gamma(M(G+\overline{K_p}))=p+\min\{\gamma(M(G[A]))~|~A\subseteq V(G), |A|=n-p\}.$$
\end{Theorem}
\begin{proof} Let $A_0\subset V(G)$ be such that $|A_0|=n-p$ and $\gamma(M(G[A_0]))=\min\{\gamma(M(G[A]))~|~A\subseteq V(G), |A|=n-p\}$. Without loss of generality, we can assume that $A_0=\{v_{p+1},\dots, v_n\}$. Let $S'$ be a minimal dominating set of $M(G[A_0])$. Then, using the same notation as in the proof of Theorem~\ref{theomindominjoinpbig}, if we consider $S=\{m_{i(n+i)}~|~1\le i\le p\}\cup S'$, we have that $S$ is a dominating set of $M(G+\overline{K_p})$ with $|S|=p+\gamma(M(G[A_0]))$, and hence $\gamma(M(G+\overline{K_p}))\le p+\min\{\gamma(M(G[A]))~|~A\subseteq V(G), |A|=n-p\}$.

On the other hand, let $S$ be a dominating set of $M(G+\overline{K_p})$. By Lemma~\ref{lemma:dominationisalledges}, we can assume $S\subseteq \mathcal{M}_1\cup \mathcal{M}_2$. Since in $G+\overline{K_p}$ each vertex of $V(\overline{K_p})$ is adjacent only to vertices in $V(G)$, we have $|S\cap\mathcal{M}_2|\ge p$. Moreover, let $V'\subset V(G)$ be the set of vertices not dominated by $S\cap\mathcal{M}_2$, then $S\cap E(G[V'])$ is a dominating set of $M(G[V'])$. This implies that $|S|\ge p+\min\{\gamma(M(G[A]))~|~A\subseteq V(G), |A|=n-p\}$, and hence $\gamma(M(G+\overline{K_p}))\ge p+\min\{\gamma(M(G[A]))~|~A\subseteq V(G), |A|=n-p\}$.

\end{proof}

%

\section{\bf Middle graph of known graphs and their domination number}\label{t800}
In this section, we obtain the domination number of the middle graph of some special families of graphs.

\begin{Proposition}\label{prop:mindominstar}
For any star graph $K_{1,n}$ on $n+1\ge2$ vertices, we have 
$$\gamma(M(K_{1,n}))=n.$$
\end{Proposition}
\begin{proof} To fix the notation, assume $V(K_{1,n})=\{v_0,v_1,\dots, v_n\}$ and $E(K_{1,n})=\{v_0v_1,v_0v_2,\dots, v_0v_n\}$. Then $V(M(K_{1,n}))=V(K_{1,n})\cup \mathcal{M}$, where $\mathcal{M}=\{ m_i~|~1\leq i \leq n \}$.

If $S=\mathcal{M}$, then $S$ is a dominating set of $M(K_{1,n})$ with $|S|=n$, and hence $\gamma(M(K_{1,n}))\le n$. On the other hand, 
since $K_{1,n}$ is a tree, we can apply Proposition~\ref{prop:mindomintreeleaf} and obtain $\gamma(M(K_{1,n}))\ge |\leaf(K_{1,n})|=n$.

\end{proof}

\begin{Theorem}\label{theo:mindominisn-1}
Let $G$ be a connected graph of order $n\ge4$. Then 
$$G=K_{1,n-1} ~~\text{ if and only if }~~ \gamma(M(G))=n-1.$$
\end{Theorem}
\begin{proof} If $G=K_{1,n-1}$, then $\gamma(M(G))=n-1$ by Proposition~\ref{prop:mindominstar}.

On the other hand, let $G$ be a connected graph such that $\gamma(M(G))=n-1$. To fix the notation, assume $V(G)=\{v_1,\dots, v_n\}$ and $V(M(G))=V(G)\cup\mathcal{M}$, where $\mathcal{M}=\{m_{ij}~|~v_iv_j\in E(G)\}$. If $G\ne K_{1,n-1}$, then $G$ has $G'$ a subgraph isomorphic to $P_4$. Without loss of generality, we can assume that $V(G')=\{v_1,\dots, v_4\}$. Then $S=\{m_{12},m_{34}\}\cup\{v_5,\dots, v_n\}$ is a dominating set of $M(G)$ with $|S|=n-2$, and hence $\gamma(M(G))\le n-2$. Hence, this is a contradiction.

\end{proof}

Putting together Theorems~\ref{theo:mindominconnectgraph} and \ref{theo:mindominisn-1}, we obtain the following result.
\begin{Corollary} Let $G$ be a connected graph of order $n\ge4$. Assume that $G\ne K_{1,n-1}$, then
$$\gamma(M(G))\le n-2.$$
\end{Corollary}

In the following result, we compute the domination number of the middle graph of the join of $K_{1,n}$ with $\overline{K_p}$. 
\begin{Proposition}\label{prop:mindominstarjoin}
For any star graph $K_{1,n}$ on $n+1\ge2$ vertices, we have 
$$\gamma(M(K_{1,n}+\overline{K_p}))=\max\{n,p\}.$$
\end{Proposition}
\begin{proof} If $p\ge n$, then the statement is a consequence of Theorem~\ref{theomindominjoinpbig}.

Assume $p<n$. To fix the notation, assume $V(K_{1,n})=\{v_0,v_1,\dots, v_n\}$ and $V(\overline{K_p})=\{v_{n+1},\dots, v_{n+p}\}$. Then $V(M(K_{1,n}+\overline{K_p}))=\{v_0,\dots\ v_{n+p}\}\cup \mathcal{M}$, where $\mathcal{M}=\{ m_i~|~1\leq i \leq n \}\cup\{m_{i(n+j)}~|~0\le i\le n, 1\le j\le p\}$. Let $S$ be a dominating set of $M(K_{1,n}+\overline{K_p})$. By Lemma~\ref{lemma:dominationisalledges}, we can assume that $S\subset\mathcal{M}$. Since, if $1\le i<j\le n$, we have $(N_{M(K_{1,n}+\overline{K_p})}(v_i)\cap\mathcal{M})\cap(N_{M(K_{1,n}+\overline{K_p})}(v_j)\cap\mathcal{M})=\emptyset$, then $|S|\ge n$. This implies that $\gamma(M(K_{1,n}+\overline{K_p}))\ge n$. On the other hand, if we consider $S=\{m_{i(n+i)}~|~1\le i\le p\}\cup\{m_i~|~p+1\le i\le n\}$, then $S$ is a dominating set of $M(K_{1,n}+\overline{K_p})$ with $|S|=n$, and hence $\gamma(M(K_{1,n}+\overline{K_p}))\le n$.

\end{proof}

Next we calculate the domination number of double star graph $S_{1,n,n}$. Notice that the graph $S_{1,n,n}$ is important because it is an example of a non-complete bipartite graph.
\begin{Definition}
A double star graph $S_{1,n,n}$ is obtained from the star graph $K_{1,n}$ by replacing every edge with a path of length $2$. 
\end{Definition}

\begin{figure}[th]
\centering
\begin{tikzpicture}
\draw (-1,0) node[v](1){}; 
\draw (-0.5,0) node[v](2){}; 
\draw (0,0) node[v](3){};
\draw (0.5,0) node[v](4){};  
\draw (1,0) node[v](5){};  
\draw (0,1) node[v](6){}; 
\draw (0,0.5) node[v](7){}; 
\draw (0,-0.5) node[v](8){};  
\draw (0,-1) node[v](9){}; 
\draw (1)--(2)--(3)--(4)--(5);
\draw (6)--(7)--(3)--(8)--(9);
\end{tikzpicture}
\caption{The double star graph $S_{1,4,4}$}\label{Fig:double star}
\end{figure}
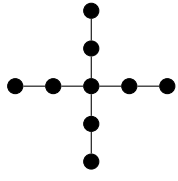

\begin{Proposition}\label{prop:mindomindoublestar}
For any double star graph $S_{1,n,n}$ on $2n+1$ vertices, with $n\ge 2$, we have $$\gamma(M(S_{1,n,n}))=n+1.$$
\end{Proposition}
\begin{proof} 
To fix the notation, assume $V(S_{1,n,n})=\{v_0,v_1,\dots, v_{2n}\}$ and $E(S_{1,n,n})=\{v_0v_i,v_iv_{n+i}~|~1\leq i \leq n\}$. Then $V(M(S_{1,n,n}))=V(S_{1,n,n})\cup \mathcal{M}$, where $\mathcal{M}=\{ m_{i},m_{i(n+i)}~|~1\leq i \leq n \}$. 

If $S=\{m_{i(n+i)}~|~1\leq i \leq n\}\cup \{v_{0}\}$, then $S$ is a dominating set of $M(S_{1,n,n})$ with $|S|=n+1$, and hence $\gamma(M(S_{1,n,n}))\le n+1$. On the other hand, by Theorem~\ref{theo:mindominconnectgraph}, $\gamma(M(S_{1,n,n}))\ge \lceil\frac{2n+1}{2}\rceil=n+1$.

\end{proof}

\begin{Proposition}\label{prop:mindominpath}
\label{gamma_{t}(M(P_n))}
For any path $P_n$ of order $n\geq 2$, we have
$$\gamma(M(P_n))=\lceil \frac{n}{2} \rceil.$$
\end{Proposition}
\begin{proof} To fix the notation, assume $V(P_n)=\{ v_1, \dots, v_n\}$ and $E(P_n)=\{v_1v_2, v_2v_3,\dots,v_{n-1}v_n\}$. Then $V(M(P_n))=V(P_n)\cup \mathcal{M}$, where $\mathcal{M}=\{ m_{i(i+1)}~|~1\leq i \leq n-1 \}$. 

Assume that $n$ is even and consider $S=\{m_{12},m_{34},\dots, m_{(n-1)n}\}$. Then $S$ is a dominating set for $M(P_n)$ with $|S|=\frac{n}{2}$. Similarly, if $n$ is odd, consider $S=\{m_{12},m_{34},\dots, m_{(n-2)(n-1)}, m_{(n-1)n}\}$. Then $S$ is a dominating set for $M(P_n)$ with $|S|=\frac{n-1}{2}+1=\lceil \frac{n}{2} \rceil$. This shows that $\gamma(M(P_n))\le\lceil \frac{n}{2} \rceil$. On the other hand, by Theorem~\ref{theo:mindomintree}, $\gamma(M(P_n))\ge\lceil \frac{n}{2} \rceil$.

%

\end{proof}

\begin{Remark} Since the star graphs and the path graphs are examples of trees, by Propositions~\ref{prop:mindominstar} and \ref{prop:mindominpath}, the inequalities of Theorems~\ref{theo:mindomintree} and \ref{theo:mindominconnectgraph} are all sharp.
\end{Remark}

Using Theorem~\ref{theo:mindominconnectgraph} and Proposition~\ref{prop:mindominpath} we obtain the following result.
\begin{Theorem}\label{theo:mindominsidepath}
Let $G$ be a graph with $n\ge2$ vertices. If $G$ has a subgraph isomorphic to $P_n$, then $$\gamma(M(G))=\lceil\frac{n}{2}\rceil.$$
\end{Theorem}
\begin{proof} By Theorem~\ref{theo:mindominconnectgraph}, we only need to show that $\gamma(M(G))\le\lceil\frac{n}{2}\rceil$.
By assumption, $G$ has $G'$ a subgraph isomorphic to $P_n$. By Lemma~\ref{lemma:dominationisalledges} and Proposition~\ref{prop:mindominpath}, there exists $S\subseteq E(G')\subseteq E(G)$ a minimal dominating set of $M(G')$ with $|S|=\lceil\frac{n}{2}\rceil$. By construction, $S$ is also a dominating set of $M(G)$, and hence $\gamma(M(G))\le\lceil\frac{n}{2}\rceil$.
\end{proof}

Notice that in general the opposite implication of Theorem~\ref{theo:mindominsidepath} is false, in fact we have the following example.
\begin{Example} Let $G$ be the graph in Figure~\ref{Fig:tree}. Then a direct computation shows that $\gamma(M(G))=3=\lceil\frac{5}{2}\rceil$, but $G$ has no subgraphs isomorphic to $P_5$.

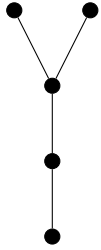
\begin{figure}[th]
\centering
\begin{tikzpicture}
\draw (0.5,0) node[v](1){}; 
\draw (0.5,1) node[v](2){}; 
\draw (0.5,2) node[v](3){};
\draw (0,3) node[v](4){};  
\draw (1,3) node[v](5){};  
\draw (1)--(2)--(3)--(4);
\draw (3)--(5);
\end{tikzpicture}
\caption{A tree on $5$ vertices}\label{Fig:tree}
\end{figure}

\end{Example}

Since any cycle graph $C_n$, any wheel graph $W_n$ and any complete graph $K_n$ contain a subgraph isomorphic to $P_n$, Theorem~\ref{theo:mindominsidepath} gives us the following result.

\begin{Corollary}\label{coroll:mindominfamilies} Let $n\ge 3$. Then 
$$\gamma(M(C_n))=\gamma(M(W_n))=\gamma(M(K_n))=\lceil\frac{n}{2}\rceil.$$
\end{Corollary}

As a consequence of Theorem~\ref{theo:mindominsidepath}, we can also compute the domination number of the middle graph of the join of a graph with $\overline{K_p}$, for several known families. 

\begin{Corollary}\label{coroll:mindominjoinfamilspecial} Let $n\ge3$ and $p<n$. If $G$ is a path graph $P_n$, a cycle graph $C_n$, a wheel graph $W_n$ or a complete graph $K_n$, then 
$$\gamma(M(G+\overline{K_p}))= \lceil\frac{n+p}{2}\rceil.$$
\end{Corollary}
\begin{proof} Under our assumption, the graph $G+\overline{K_p}$ contains a subgraph isomorphic to $P_{n+p}$, and hence Theorem~\ref{theo:mindominsidepath} gives us the result.

\end{proof}

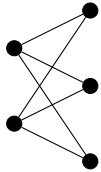
\begin{figure}[th]
\centering
\begin{tikzpicture}
\draw (0.5,1) node[v](1){}; 
\draw (-0.5,0.5) node[v](2){}; 
\draw (0.5,0) node[v](3){};
\draw (-0.5,-0.5) node[v](4){};  
\draw (0.5,-1) node[v](5){};  
\draw (1)--(2);
\draw (1)--(4);
\draw (3)--(2);
\draw (3)--(4);
\draw (5)--(2);
\draw (5)--(4);
\end{tikzpicture}
\caption{The complete bipartite graph $K_{2,3}$}\label{Fig:compbipartite}
\end{figure}

\begin{Proposition}\label{prop:mindomincompletebipartitegr}
Let $K_{n_1,n_2}$ be the complete bipartite graph with $n_2\geq n_1 \geq 1$. Then $$\gamma(M(K_{n_1,n_2}))= n_2.$$
\end{Proposition}
\begin{proof}  To fix the notation, assume $V(K_{n_1,n_2})=\{v_1,\dots, v_{n_1},u_1,\dots, u_{n_2}\}$ and $E(K_{n_1,n_2})=\{v_iu_j ~|~1\leq i \leq n_1, 1\leq j \leq n_2\}$. Then $V(M(K_{n_1,n_2}))=V(K_{n_1,n_2})\cup \mathcal{M}$, where
 $\mathcal{M}=\{ m_{ij}~|~1\leq i \leq n_1, 1\leq j \leq n_2 \}. $
Consider  $S=\{m_{11},m_{22},\dots, m_{n_1n_1}\}\cup\{m_{n_1(n_1+1)}, m_{n_1(n_1+2)},\dots, m_{n_1n_2}\}$.
By construction, $S$ is a dominating set of $M(K_{n_1,n_2})$ and $|S|=n_1+n_2-n_1=n_2$. This implies that $\gamma(M(K_{n_1,n_2}))\le n_2$.

On the other hand, if $S$ is a dominating set for $M(K_{n_1,n_2})$, then it dominates $u_1,\dots, u_{n_2}$ that are all disconnected. This implies that $\gamma(M(K_{n_1,n_2}))\ge n_2$.
\end{proof}

Notice that if we consider the case when $n_1=1$, the previous result gives us Proposition~\ref{prop:mindominstar}.

\begin{Definition} The \emph{friendship} graph $F_n$ of order $2n+1$ is obtained by joining $n$ copies of the cycle graph $C_3$ with a common vertex.
\end{Definition}

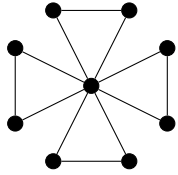
\begin{figure}[th]
\centering
\begin{tikzpicture}
\draw (0,0) node[v](1){}; 
\draw (-1,0.5) node[v](2){}; 
\draw (-1,-0.5) node[v](3){};
\draw (1,0.5) node[v](4){};  
\draw (1,-0.5) node[v](5){};  
\draw (-0.5,1) node[v](6){};  
\draw (0.5,1) node[v](7){}; 
\draw (-0.5,-1) node[v](8){};  
\draw (0.5,-1) node[v](9){}; 
\draw (1)--(2)--(3)--(1);
\draw (1)--(4)--(5)--(1);
\draw (1)--(6)--(7)--(1);
\draw (1)--(8)--(9)--(1);
\end{tikzpicture}
\caption{The friendship graph $F_{4}$}\label{Fig:friendship}
\end{figure}

\begin{Proposition}\label{prop:mindominfriendship}
Let $F_n$ be the friendship graph with $n\ge2$. Then $$\gamma(M(F_n))= n+1.$$
\end{Proposition}
\begin{proof} To fix the notation, assume $V(F_n)=\{v_0,v_1,\dots, v_{2n}\}$ and $E(F_n)=\{v_0v_1,v_0v_2,\dots, v_0v_{2n}\}\cup\{v_1v_2, v_3v_4,\dots,v_{2n-1}v_{2n}\}$. Then $V(M(F_n))=V(F_n)\cup \mathcal{M}$, where $\mathcal{M}=\{ m_i~|~1\leq i \leq 2n \}\cup\{ m_{i(i+1)}~|~1\leq i \leq 2n-1 \text{ and } i \text{ is odd}\}$. 

Consider $S=\{ m_{i(i+1)}~|~1\leq i \leq 2n-1 \text{ and } i \text{ is odd}\}\cup\{v_0\}$. Then $S$ is a dominating set for $M(F_n)$ with $|S|=n+1$, and hence, $\gamma(M(F_n))\le n+1$. On the other hand, by Theorem~\ref{theo:mindominconnectgraph}, $\gamma(M(F_n))\ge \lceil\frac{2n+1}{2} \rceil=n+1$.

\end{proof}

\section{Nordhaus-Gaddum relations}

In \cite{Nordhaus}, Nordhaus and Gaddum gave lower and upper bounds on the sum and the product of the chromatic number of a graph and its complement, in terms of the order of the graph. 
Since then, lower and upper bounds on the sum and the product of many other graph invariants, like domination and total domination numbers, have been proposed by several authors. See \cite{Aouchiche13} for a survey on the subject.


In this section, we study Nordhaus-Gaddum relations for the domination number of middle graphs.

%
%
%

\begin{Theorem}\label{theo:Nordgaddomin}
Let $G$ be a graph on $n\ge2$ vertices. Assume that the graphs $G$ and $\overline{G}$ have no isolated vertices. Then
$$2(n-2)\ge \gamma(M(G))+\gamma(M(\overline{G}))\ge 2\lceil\frac{n}{2}\rceil $$
and
$$(n-2)^2\ge \gamma(M(G))\cdot\gamma(M(\overline{G}))\ge (\lceil\frac{n}{2}\rceil)^2. $$
\end{Theorem}
\begin{proof} By Theorem~\ref{theo:mindominisn-1}, if $\gamma(M(G))=n-1$, then $G=K_{1,n-1}$ and hence, $\overline{G}$ has an isolated vertex. Similarly, if $\gamma(M(\overline{G}))=n-1$, then $G$ has an isolated vertex. This fact implies that, under our assumptions, $(n-2)\ge \gamma(M(G))$ and $(n-2)\ge \gamma(M(\overline{G}))$. On the other hand, Theorem~\ref{theo:mindominconnectgraph} implies that  $\gamma(M(G))\ge \lceil\frac{n}{2}\rceil $ and $\gamma(M(\overline{G}))\ge \lceil\frac{n}{2}\rceil $.

\end{proof}

Notice that all the inequalities of Theorem~\ref{theo:Nordgaddomin} are sharp, in fact we have the following example.
\begin{Example}
Consider the graph $C_4$ in Figure~\ref{Fig:NGC4}. Then by Corollary~\ref{coroll:mindominfamilies}, we have $\gamma(M(C_4))=2$. On the other hand, $\overline{C_4}=2K_2$, i.e. it is two distinct copies of the graph $K_2$, and hence, by Corollary~\ref{coroll:mindominfamilies}, we have $\gamma(M(\overline{C_4}))=2$. Since $n=4$, then $4=\gamma(M(C_4))+\gamma(M(\overline{C_4}))=\gamma(M(C_4))\cdot\gamma(M(\overline{C_4}))=2(n-2)=2\lceil\frac{n}{2}\rceil=(n-2)^2=(\lceil\frac{n}{2}\rceil)^2$.
\end{Example}

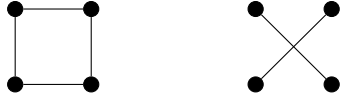
\begin{figure}[th]
\centering
\subfigure
{
\begin{tikzpicture}
\draw (0.5,0.5) node[v](1){}; 
\draw (-0.5,0.5) node[v](2){}; 
\draw (-0.5,-0.5) node[v](3){};
\draw (0.5,-0.5) node[v](4){};  
\draw (1)--(2)--(3)--(4)--(1);
\end{tikzpicture}
}
\hspace{14mm}
\subfigure
{
\begin{tikzpicture}
\draw (0.5,0.5) node[v](1){}; 
\draw (-0.5,0.5) node[v](2){}; 
\draw (-0.5,-0.5) node[v](3){};
\draw (0.5,-0.5) node[v](4){};  
\draw (1)--(3);
\draw (2)--(4);
\end{tikzpicture}
}
\caption{The graphs $C_4$ and $\overline{C_4}$}\label{Fig:NGC4}
\end{figure}

Notice that the results of Section 3 allow us to compute exactly $\gamma(M(G))+\gamma(M(\overline{G}))$ and $\gamma(M(G))\cdot\gamma(M(\overline{G}))$ for several graphs, as the following examples show.

\begin{Example} Consider the graph $G=K_{1,n}$ with $n+1\ge2$. Then $\overline{G}=K_n\cup\overline{K_1}$. By Proposition~\ref{prop:mindominstar} and Corollary~\ref{coroll:mindominfamilies}, $\gamma(M(G))+\gamma(M(\overline{G}))=n+\lceil\frac{n}{2}\rceil+1$ and $\gamma(M(G))\cdot\gamma(M(\overline{G}))=n(\lceil\frac{n}{2}\rceil+1)$.
\end{Example}

\begin{Example} Consider the graph $G=K_n$ with $n\ge2$. Then by Corollary~\ref{coroll:mindominfamilies}, $\gamma(M(G))+\gamma(M(\overline{G}))=\lceil\frac{n}{2}\rceil+n$ and $\gamma(M(G))\cdot\gamma(M(\overline{G}))=n(\lceil\frac{n}{2}\rceil)$.
\end{Example}

\begin{Example} Consider the graph $G=P_n$ with $n\ge2$. The graph $\overline{G}$ has a subgraph isomorphic to $P_n$. By Proposition~\ref{prop:mindominpath} and Theorem~\ref{theo:mindominsidepath}, $\gamma(M(G))+\gamma(M(\overline{G}))=2(\lceil\frac{n}{2}\rceil)$ and $\gamma(M(G))\cdot\gamma(M(\overline{G}))=(\lceil\frac{n}{2}\rceil)^2$.
\end{Example}

%
%

\paragraph{\textbf{Acknowledgements}} During the preparation of this article the fourth author was supported by JSPS Grant-in-Aid for Early-Career Scientists (19K14493).


\end{document}